\documentclass[12pt,reqno]{amsart}
\usepackage{amssymb}
\usepackage{epsf}
\usepackage{amsmath,epsf}
\usepackage[latin1]{inputenc}
\usepackage{amsfonts}

\usepackage{amssymb}
\usepackage{amsmath}

\usepackage{enumerate}
\usepackage{verbatim}
\usepackage[colorlinks=true,citecolor=cyan]{hyperref}

\usepackage{graphicx}

\usepackage{ifpdf}
\ifpdf
\fi

\newcommand{\pref}[1]{{\rm (\ref{#1})}}

\def\jaune{\textcolor{yellow}}
\def\bleu{\textcolor{blue}}
\def\rouge{\textcolor{red}}
\def\vert{\textcolor{green}}

\def\auteur#1{{\sc #1}}
\def\titreref#1{{\em #1}}
\def\vol#1{{\bf #1}}

\newtheorem{theorem}{\bleu{Theorem}}[section]
\newtheorem{lemma}[theorem]{\bleu{Lemma}}
\newtheorem{proposition}[theorem]{\bleu{Proposition}}

\newtheorem{definition}[theorem]{Definition}

\newtheorem{conjecture}[theorem]{\bleu{Conjecture}}

\numberwithin{equation}{section}

\newcommand{\monem}[1]{\emph{\bleu{#1}}}
\newcommand{\D}{\mathcal{D}}
\newcommand{\poly}{\mathcal{R}_n^{(\ell)}}
\newcommand{\colored}{\mathcal{C}}
\renewcommand{\S}{\mathbb{S}}

\newcommand{\NN}{\mathbb{N}}

\newcommand{\QQ}{\mathbb{Q}}

\newcommand{\bz}{{\bf 0}}
\newcommand{\Gpol}[1]{G\big[ \begin{smallmatrix}  #1  \end{smallmatrix}\big]}
\newcommand{\Mpol}[1]{M\big[ \begin{smallmatrix}  #1  \end{smallmatrix}\big]}

\newcommand{\JW}[1][]{J_{#1}}
\newcommand{\tr}{{\rm tr}}

\newcommand{\DQcoinv}[1][W]{\mathcal{Q}_{#1}}
\newcommand{\DQcoinvd}[1][d]{\mathcal{Q}_{W,#1}}
\newcommand{\DQW}[1][W]{\mathcal{S}_{#1}}
\newcommand{\Id}{\mathrm{Id}}

\newcommand{\p}{\partial}

\newcommand{\X}{X}
\newcommand{\PX}[1][\X]{\partial_{#1}}

\newcommand{\mq}[1][q]{{\mathbf{#1}}}

\usepackage{ifthen}

\title[co-quasi-invariant spaces]{\bleu{Co-quasi-invariant spaces\\ 
for finite\\ complex reflection groups}}

\author{J.-C. Aval} 
\address{LaBRI, Universit\'e de Bordeaux 1, CNRS, 351 cours de la Lib\'eration, 33405 Talence, France.}
\author{F. Bergeron}
\address{D\'epartement de Math\'ematiques, UQAM,  C.P. 8888, Succ. Centre-Ville, 
 Montr\'eal,  H3C 3P8, Canada.}

\begin{document}
\maketitle

\begin{abstract} We study, in a global uniform manner, the quotient of the ring of polynomials in $\ell$ sets of $n$ variables, by the ideal generated by diagonal quasi-invariant polynomials for general permutation groups $W=G(r,n)$. We show that, for each such group $W$, there is an explicit universal symmetric function that gives the $\NN^\ell$-graded Hilbert series for these spaces. This function is universal in
that its dependance on $\ell$ only involves the number of variables it is calculated with. We also discuss the combinatorial implications of the observed fact that it affords an expansion as a positive coefficient polynomial in the complete homogeneous symmetric functions.
 
\end{abstract}


 {
   \setcounter{tocdepth}{1}
   \parskip=0pt
   \footnotesize
   \tableofcontents
 }

\section{Introduction}

For  rank $n$ classical families of finite complex reflection groups $W$, we contribute to the description of the  diagonal co-quasi-invariant space $\DQcoinv$ for $W$, in several (say $\ell$)  sets  of $n$ variables. Here, the use of the term diagonal refers to the fact that $W$ is considered as a diagonal subgroup of $W^\ell$, acting on the $\ell^{\rm th}$-tensor power $\poly$ of the symmetric algebra of the defining representation of $W$. 
Instead of the usual one, the action considered here is the so-called Hivert-action. Invariant polynomials under this action are known as quasi-invariants\footnote{Under the same name, an entirely different notion has been considered in~\cite{etingof, feigin}. However, the terminology of quasi-symmetric polynomials being well ingrained, it seems awkward to call their generalization to other reflection groups by any other name then quasi-invariant.} (or quasi-symmetric for the symmetric group).
Our space $\DQcoinv^{(\ell)}$ is simply the quotient of $\poly$ by the ideal generated by constant-term-free quasi-invariants for $W$.
We show that the associated multigraded Hilbert series, denoted $\DQcoinv^{(\ell)}(q_1,\ldots,q_\ell)$ (which is symmetric in the $q_i$),  can be described in an uniform manner as a positive coefficient linear combination of Schur polynomials
    \begin{equation}\label{schur_generic}
              \bleu{\DQcoinv^{(\ell)}(q_1,\ldots,q_\ell) = \sum_\mu c_\mu\, s_\mu(q_1,\ldots,q_\ell),}
       \end{equation}
 with the $c_\mu$ independent of $\ell$, and $\mu$ running through a finite set of integer partitions that depend only on the group $W$. This a typical phenomena in many similar situations such as considered in~\cite{bergeron_several}. It has the striking feature that we can give explicit formulas for the dimension of $\DQcoinv^{(\ell)}$ for all $\ell$. To see why this is so striking, it may be worthwhile to recall that, for the entirely analogous context of diagonal co-invariant spaces (i.e. the one corresponding to the usual diagonal action of $W$ on $\poly$), a large body of work has only recently settled the special case $\ell=2$, but that we know almost nothing yet for $\ell\geq 3$.

\section{Our context}
A down to earth description of our context may be given as follows.
Consider a  $\ell \times n$ matrix of variables $X:=(x_{ij})$. For any fixed $i$ (a row of $X$), we say that the variables \bleu{$x_{i1},x_{i2},\ldots, x_{in}$} form the $i^{\rm th}$ \monem{set of variables}.
In some instances it is worth simplifying this notation, and write
   $$X= \begin{pmatrix}  
         x_1 & x_2 & \cdots & x_n\\
         y_1 & y_2 & \cdots & y_n\\
         \vdots &\vdots &\ddots &\vdots\\
          z_1 & z_2 & \cdots & z_n\\
              \end{pmatrix} .$$
 Thus, $x=x_1,\ldots,x_n$ stands for the first set of variables, $y=y_1,\ldots,y_n$ for the second set, $\ldots$, and $z=z_1,\ldots,z_n$ for the last set.
 
 With the aim of certain describing polynomials in the variables $X$, we choose to  denote by $X_j$ the $j^{\rm th}$ column of $X$, for $1\leq j\leq n$. We assume the same convention for any $\ell \times n$ matrix of non-negative integers $A$. Moreover, if the $A_i$ are the columns of $A$, we write
     $$\bleu{A=A_1A_2\cdots A_n}.$$
We then consider the monomials
     $$X_j^{A_j}:=\prod_{i=1}^\ell x_{ij}^{a_{ij}},\qquad \hbox{as well as}\qquad X^A:=\prod_{j=1}^n X_j^{A_j}.$$
For the monomials $X^A$, who clearly form a basis of space of polynomials $\poly:=\QQ[X]$, the corresponding \monem{degree vector}:
     $$\deg(X^A):=\sum_j A_j,$$
lies in $\NN^\ell$.

Given $r,n\in \NN^+$, recall that the
\monem{generalized symmetric group} $W=G(r,n)$ may be described as the
group of $n\times n$ matrices having exactly one non zero coefficient in
each row and each column, which  is a $r^{\rm th}$ root of unity.
One usually considers $W$ as acting on polynomials in $\poly=\QQ[X]$  by 
replacement of the variables by the matrix $X\,w$. 
With this point of view, we may consider that $W$ is generated by the \monem{transpositions} $s_j$, which exchange columns $j$ and $j+1$ in $X$, together with $s_0$ which multiplies the first column of $X$ by a (chosen) primitive $r^{\rm th}$ root of unity. 
These generators $s_j$ satisfy the usual Coxeter relations for $j\geq 1$:
\begin{eqnarray*}
    s_j^2&=&\Id,\qquad (s_j\,s_{j+1})^3=\Id,\quad\mathrm{and}\\
     s_j\,s_k&=&s_k\,s_j,\quad\mathrm{when}\quad |j-k|>1.
  \end{eqnarray*}
For the special pseudo-reflection $s_0$,  we have $s_0^r=\Id$ and $(s_0\,s_1)^{2\,r}=\Id$. 
This is the diagonal action which is considered for the ``usual'' definition of the diagonal co-invariant space for $W$ (see~\cite{livre}). Rather than this space, we consider a variant below.

Our point of departure from the ``classical'' situation is to consider rather the diagonal quasi-invariant polynomials for $W$. For example, taking $W=G(1,3)$ (the symmetric group $\mathbb{S}_3$) and $\ell=1$, we have the quasi-invariant (or quasi-symmetric) polynomials:
$$\begin{array}{lll}
x_{{1}}+x_{{2}}+x_{{3}},&
{x_{{1}}}^{2}+{x_{{2}}}^{2}+{x_{{3}}}^{2},&
{x_{{1}}}^{3}+{x_{{2}}}^{3}+{x_{{3}}}^{3},\\
x_{{1}}x_{{2}}+x_{{1}}x_{{3}}+x_{{2}}x_{{3}},&
x_{{1}}{x_{{2}}}^{2}+x_{{1}}{x_{{3}}}^{2}+x_{{2}}{x_{{3}}}^{2},&
{x_{{1}}}^{2}x_{{2}}+{x_{{1}}}^{2}x_{{3}}+{x_{{2}}}^{2}x_{{3}}.\\
\end{array}$$
For $\ell=2$, another  $\mathbb{S}_3$-quasi-invariant  is the polynomial
    $$y_{{1}}x_{{2}}y_{{2}}+y_{{1}}x_{{3}}y_{{3}}+y_{{2}}x_{{3}}y_{{3}}.$$
The vector space of diagonally quasi invariant for $W=G(r,n)$ is spanned by the \monem{monomial basis} $\{M_A\}_{A\in \mathcal{B}_{r,n}}$, which are  indexed by \monem{$r$-composition-matrices}. These are the positive integer entries $\ell\times k$ matrices, $1\leq k\leq n$,  having all column sums congruent to $0$ mod $r$, with no column sum actually vanishing. We  say that we have an \monem{$r$-matrix} if this last condition is dropped.
The monomial quasi-invariant associated to such a $r$-composition-matrix is simply defined as
   $$M_A:=\sum_{Y\subseteq X} Y^A,$$
with $Y$ running over all matrices obtained by selecting (in the order that they appear) $k$ columns of $X$. We sometimes write $M[A]$ for $M_A(X)$. It is easy to check directly that this is indeed
a basis. For example,  we have
$$M\Big[ \begin{smallmatrix}  
       1&3\\
       0&1\\
       2&0  
             \end{smallmatrix}\Big]=\sum_{a<b} x_{a}\,z_{a}^2\,x_{b}^3\,y_{b}.$$              
We simply denote by \bleu{$\JW[W]$} (or often simply by \bleu{$\JW$}) the ideal generated by constant-term-free diagonally $W$-quasi-invariant polynomials. 

We may now define our main object of study, which is the \monem{co-quasi-invariant} space:
\begin{equation}\label{def_dqcoinv}
    \DQcoinv^{(\ell)}:=\poly/\JW[W].
  \end{equation}
 To better analyze the structure of this space, we need to consider the action of the general linear group $GL_\ell$ on $\poly$, defined by
  \begin{equation}\label{action_GL}
    (f\cdot \tau)(X):=f(\tau\, X),\qquad {\rm for} \quad \tau \in  GL_\ell.
  \end{equation}
 Observing that the ideal $\JW=\JW[W]$ is invariant under this action, we conclude that $ \DQcoinv^{(\ell)}$ inherits a $GL_\ell$-module structure.
 Since the ideal $\JW$ is homogeneous for the vector-degree, $ \DQcoinv^{(\ell)}$ may be graded by this same vector-degree, i.e.:
     $$ \DQcoinv^{(\ell)} =\bigoplus_{d\in \NN^\ell} \DQcoinvd,$$
  with $\DQcoinvd$ denoting the homogeneous component of degree $d$ of $\DQcoinv^{(\ell)}$.
 It follows that the associated \monem{Hilbert series}, \bleu{ $ \DQcoinv^{(\ell)}(\mq)$}, coincides with the character of $ \DQcoinv^{(\ell)}$ as a $GL_\ell$-module. To help the reader parse this statement, let us assume that $\mathcal{B}$ is a basis consisting of homogeneous elements of $ \DQcoinv^{(\ell)}$. This means that, for $f(X)$ in $ \mathcal{B}$, we have
\begin{equation}\label{defn_homogene}
   f( \mq\, X) = \mq^d\, f(X),
\end{equation}
 where $\mq$ stands for the diagonal matrix 
    $$\mq=\begin{pmatrix} q_1 \\ 
                           &\ddots\\ 
                             &&q_\ell \end{pmatrix},$$
and $\mq^d:=q_1^{d_1}\cdots q_\ell^{d_\ell}$. Thus, an homogeneous $f(X)$ is an eigenvector of the linear transform $\mq^*$ sending $f(x)$ to $f(\mq\,X)$. Recall here that, by definition, the trace of $\mq^*$, as a function of the $q_i$, is the character of $ \DQcoinv^{(\ell)}$.
Summing up, the Hilbert series of the space, defined by the expression    
 \begin{equation}
      \DQcoinv^{(\ell)}(\mq):=\sum_{d\in \NN^\ell} \mq^d\,\dim(\DQcoinvd),
  \end{equation}
coincides with the (also usual) definition of the character of the corresponding (polynomial) representation of $GL_\ell$.

The point of this last observation is that  $ \DQcoinv^{(\ell)}(\mq)$  is \monem{Schur-positive}, since Schur functions $s_\mu(\mq)$ appear as characters of irreducible representations of $GL_\ell$. Indeed, the decomposition into irreducibles of the polynomial $GL_\ell$-representation $ \DQcoinv^{(\ell)}(\mq)$
gives a formula of the form~\pref{schur_generic}, with $\mu$ running through all partitions for which the homogeneous component $\DQcoinvd[\mu]$ is non-vanishing. For example,
 for the symmetric group, one finds the following expressions for $ \DQcoinv[n]:= \DQcoinv[{\mathbb{S}_n}]^{(\ell)}\!(\mq)$
\begin{eqnarray*}
 \DQcoinv[1]&=&1,\\
 \DQcoinv[2]&=&1+s_{{1}}(\mq),\\
 \DQcoinv[3]&=&1+2\,s_{{1}}(\mq)+2\,s_{{2}}(\mq),\\
 \DQcoinv[4]&=&1+3\,s_{{1}}(\mq)+5\,s_{{2}}(\mq)+2\,s_{{11}}(\mq)+5\,s_{{3}}(\mq),\\
 \DQcoinv[5]&=&1+4\,s_{{1}}(\mq)+9\,s_{{2}}(\mq)+5\,s_{{11}}(\mq)+14\,s_{{3}}(\mq)\\
      &&\qquad\qquad\qquad\qquad\qquad+10\,s_{{21}}(\mq)+14\,s_{{4}}(\mq).
\end{eqnarray*}
These examples exhibit the announced striking ``independence'' with respect to $\ell$. \
 
Before going on with our discussion, let us introduce another $GL_\ell$-module  which is isomorphic (both as a  $GL_\ell$-module and a $W$-module) to  the space $\DQcoinv^{(\ell)}$. For each of the variables $x_{ij} \in X$, consider the
partial derivation denoted by $\p_{x_{ij}}$, or $\p_{ij}$ for
short. For a polynomial $f(X)$, we then denote by $f(\PX)$  the
differential operator obtained by replacing the variables in $\X$ by the
corresponding derivation in $\PX$.  The space \bleu{$\DQW^{(\ell)}$} of
\monem{diagonally super-harmonic polynomials} with respect to $W$-quasi-invariants  is simply defined to be the set of 
polynomial solutions $g(X)$ of the system of partial differential
equations
\begin{equation}\label{defn_harmonic}
    f(\PX ) (g(X)) = 0,\qquad \hbox{for}\quad f(X)\in \JW\,.
\end{equation}
Evidently, we need only consider a generating set of $\JW$ for these equations to characterize all solutions. 
The elementary proof (see~\cite{livre}) that $\DQcoinv^{(\ell)}$ and $\DQW^{(\ell)}$ are isomorphic relies on the fact that there is a scalar product for which $\DQW^{(\ell)}$ appears as the orthogonal complement of $\JW[W]$.

Following our established conventions, \bleu{$\DQW^{(\ell)}(\mq)$} stands for the \monem{Hilbert series} of the graded space $\DQW^{(\ell)}$. From the above discussion, this is equal to the Hilbert series $\DQcoinv^{(\ell)}(\mq)$.  The advantage of working with $\DQW^{(\ell)}$ is that we may present a basis in terms of explicit polynomials (which give canonical representatives for equivalence classes in $\DQcoinv^{(\ell)}$).

To get a better feeling of how things work out, let us first consider the case $W=\S_3$ and $\ell=2$. We may then check that we have the following  bases $\mathcal{B}_d$ for the various homogeneous components $\mathcal{H}^{(2)}_{3,d}$ of the space $\mathcal{H}^{(2)}_{3}=\DQW^{(2)}$.
$$\begin{array}{rcl}
\mathcal{B}^{(2)}_{00}&=&\{ 1 \}, \\
\mathcal{B}^{(2)}_{10}&=&\{ -x_{{1}}+x_{{2}},-x_{{1}}+x_{{3}} \}, \\
\mathcal{B}^{(2)}_{01}&=& \{ -y_{{1}}+y_{{2}},-y_{{1}}+y_{{3}} \}, \\
\mathcal{B}^{(2)}_{20}&=& \{ - \left( x_{{1}}-x_{{2}} \right)  \left( x_{{1}}
-2\,x_{{3}}+x_{{2}} \right) ,\\
 &&\qquad - \left( x_{{1}}-x_{{3}} \right)  \left( x_{{1}}-2\,x_{{2}}+x_{{3}} \right)  \}, \\
 \mathcal{B}^{(2)}_{11}&=&\{ x_{{2}}y_{{2}}-x_{{1}}y_{{2}}-x_{{3}}y_{{3}}+x_{{1}}y_{{3}}
          -y_{{1}}x_{{2}}+y_{{1}}x_{{3}},\\
 &&\qquad x_{{2}}y_{{3}}-x_{{3}}y_{{3}}+x_{{1}}y_{{1}}-y_{{1}}x_{{2}}-x_{{1}}y_{{2}}+y_{{2}}x_{{3}} \}, \\
\mathcal{B}^{(2)}_{02}&=& \{ - \left( y_{{1}}-y_{{2}} \right)  \left( y_{{1}}
-2\,y_{{3}}+y_{{2}} \right) ,\\
 &&\qquad  - \left( y_{{1}}-y_{{3}} \right) 
 \left( y_{{1}}-2\,y_{{2}}+y_{{3}} \right)  \} .
 \end{array}$$
Observe that we can calculate $\mathcal{B}^{(2)}_{0k}$ from  $\mathcal{B}^{(2)}_{k0}$ by exchanging all $x_i$ by the corresponding $y_i$. 

By contrast, for $\ell=3$, the space  $\DQW[3]^{(3)}=\DQW^{(3)}$ affords the following bases. For all $d$ of the form $jk0$, we may choose
$\mathcal{B}^{(3)}_{jk0}:=\mathcal{B}^{(2)}_{jk}$. To get the bases for the other non-vanishing components of  $\DQW^{(3)}$, we set
$\mathcal{B}^{(3)}_{j0k}$ equal to the set  of polynomials obtained by exchanging the $y_i$ by the corresponding $z_i$ for all elements of $\mathcal{B}^{(3)}_{jk0}$. In turn, we get $\mathcal{B}^{(3)}_{0jk}$ from $\mathcal{B}^{(3)}_{j0k}$, now exchanging the $x$-variables for the $y$-variables. It can then be checked that there are no other non-vanishing component in $\mathcal{H}^{(3)}_{3}$. One may also use Theorem~\ref{mainthm} to see this.
Our point here is that we get the two Hilbert series
\begin{eqnarray*}
  \mathcal{H}^{(2)}_{3}(q_1,q_2)&=&1+2\,(q_1+q_2) + 2\,(q_1^2+q_2^2+q_1q_2)\\
  \mathcal{H}^{(3)}_{3}(q_1,q_2,q_3)&=&1+2\,(q_1+q_2+q_3) + 2\,(q_1^2+q_2^2+q_3^2+q_1q_2+q_1q_3+q_2q_3)
\end{eqnarray*}
both taking the form $ \mathcal{H}_{3}(\mq)=1+2\,s_1(\mq)+2\,s_2(\mq)$, as announced.

\section{General results}
\begin{theorem}\label{mainthm}
   For any given complex reflection group $W=G(r,n)$,  the Hilbert series $\DQcoinv^{(\ell)}(\mq)$ affords an expansion in terms of Schur functions,  with       
     positive integer coefficients  that are independent of $\ell$, the sum being over the set of partitions of integers $d$:
     \begin{equation}\label{interval_entiers}
      \bleu{ 0\leq d \leq 2\,r\,n-r -n},
     \end{equation}
   and having at most $\bleu{n}$ parts.
\end{theorem}
To better underline one of the most important feature of this statement, we may consider that the symmetric function involved in these expressions are written in terms of infinitely many variables 
    $$\mq=q_1,q_2,q_3,\, \ldots$$
This makes formula~\pref{schur_generic} entirely independent of $\ell$.
To get the Hilbert series in the special case of $\ell$ sets of $n$ variables,  we  simply specialize this ``universal'' formula by setting all variables $q_k$, for $k>\ell$,  equal to zero. This process is made even more transparent by ``removing the variables'', writing $s_\mu$ (or $h_\mu$) instead of $s_\mu(\mq)$ (or $h_\mu(\mq)$)\footnote{As in Macdonald~\cite{macdonald}, we write $h_k$ for the complete homogeneous symmetric functions.} in \pref{schur_generic}. In other words, we consider $f(q_1,\ldots,q_\ell)$ as the evaluation, of a (variable free) symmetric functions $f$, inn the set of $\ell$ variables $q_1,\ldots, q_\ell$. We may also drop the $\ell$ in $\DQcoinv^{(\ell)}$.

The following formula, for the case $\ell=1$, is shown to hold in~\cite{aval_W}. Namely, for the group $W=G(r,n)$, we have
\begin{equation}
  \bleu{ \DQcoinv(q,0,0,\ldots)= \left(\frac{1-q^r}{1-q}\right)^{\!\!n} \cdot \sum_{k=0}^{n-1}\frac{n-k}{n+k}\binom{n+k}{k}\, q^{r\,k}}.
\end{equation} 
In particular, for $n=2$, we get
   $$ \DQcoinv(q,0,0,\ldots)=(1+q+\ldots+q^{r-1})^2\, (1+q^r).$$
Since $h_k(q,0,0,\ldots)=q^k$, this is readily seen to be the specialization at 
     $$\mq=q,0,0,\ldots$$ 
 of the following ``universal'' formula (see~\cite{bergeron_several}), for the groups $W=G(r,2)$:
      \begin{equation}\label{formule_Gr2}
         \bleu{ \left( \sum _{k=0}^{r-1}h_{{k}} \right) ^{\!\!2}+\sum _{k=0}^{r-1} \left( k+1 \right) h_{{r+k}}+\sum _{k=1}^{r-1} \left( r-k \right) h_{{2\,r-1+k}}}.
       \end{equation}
 As such, it holds for the diagonal co-invariant space (under the classical action) which, in this very specific case, coincides with the space of diagonal co-quasi-invariant space.

A nice feature of the expression given in~\pref{formule_Gr2} is its \monem{$h$-positivity}:
    $$\sum_\mu a_\mu\, h_\mu,\qquad {\rm with}\quad a_\mu\geq 0.$$
 This appears to hold for many other reflection groups, in particular  when $W$ is a symmetric group, leading us to state the following.
 \begin{conjecture}\label{conj}
    For the symmetric group, the  Hilbert series $\DQcoinv[n](\mq)$  is $h$-positive in degrees smaller than $n/2$.
  \end{conjecture}
An immediate consequence of this conjecture is that $\DQcoinv[n](\mq)$  has to have a very specific form, since 
    \begin{eqnarray}
         \DQcoinv[n](q,0,0,\ldots)&=&\sum_\mu a_\mu\, h_\mu(q,0,0,\ldots)\nonumber \\
                            &=&\sum_\mu a_\mu\,q^{|\mu|}\nonumber \\
                            &=&\sum_{k=0}^{n-1}\frac{n-k}{n+k}\binom{n+k}{k}\, q^{k},\label{formule_q1}\\
                            &=&\sum_\beta q^{\chi(\beta)},\label{formule_q2}
    \end{eqnarray}
with $\beta$ running over the set of Dyck paths\footnote{See section~\ref{sec_colored} for more details.} of height $n$, and $\chi(\beta)$ taking as value the $x$-coordinate of the first point of the path at height $n$. 
The passage from \pref{formule_q1} to \pref{formule_q2} is classical.
It follows that 
\begin{proposition}\label{mainprop}
  Conjecture~\ref{conj} implies that 
$ \DQcoinv[n](\mq)$  affords an expression of the form
     \begin{equation}\label{formule_DQcoinv}
          \bleu{\DQcoinv[n](\mq)=_{\leq n/2} \sum_\beta h_{\mu(\beta)}(\mq)},
       \end{equation}
   with $\beta$  running over the set of all Dyck paths of height $n$, and $\mu(\beta)$ some partition of the integer $\chi(\beta)$. Here $=_{\leq n/2}$ stands for equality in degrees less or equal to $n/2$.
\end{proposition}
  As of this writing, we do not have a rule for producing the partition $\mu(\beta)$ associated to $\beta$, which would haver to be compatible with the actual values given in~\pref{valeurs_pour_Sn}. 
  
\section{Formulas for low degree components}
We discuss now how to get explicit polynomial formulas in the variable $n$ for the coefficient of $h_\mu$, when $\mu$ is a partition of a small enough integer. 
We restrict the discussion to the case $W=\S_n$, but much of it holds in generality. 
We exploit here the fact that low degree homogeneous components of the spaces $\mathcal{R}_n$
and $ \DQcoinv[n]\otimes \mathcal{R}_n^{\sim \S_n}$ are isomorphic, where $ \mathcal{R}_n^{\sim \S_n}$ stands for the ring of diagonal quasi-symmetric polynomials. 
This immediately implies that we have explicit formulas for the relevant homogeneous components of   $ \DQcoinv[n]$, since we have the explicit expressions
\begin{equation}
    \bleu{\mathcal{R}_n(\mq)= (1+H(\mq))^n},\qquad {\rm and}\qquad \bleu{\mathcal{R}_n^{\sim \S_n}(\mq)= \frac{1}{1-H(\mq)}},
\end{equation}
where $H(\mq):=\sum_{k\geq 1} h_k(\mq)$.
It follows that we may calculate the low degree terms of the Hilbert series $\DQcoinv[n](\mq)$, via the expansion
  \begin{equation}\label{formule_hilb}
  \begin{array}{rcl}
      \bleu{\displaystyle (1+H(\mq))^n  (1-H(\mq))}&=&\bleu{ 1+(n-1) h_1+(n-1) h_2+\frac{1}{2}\,n(n-3) h_1^2+(n-1) h_3}\\[6pt]
         &&\qquad \bleu{+\frac{1}{6}\,n(n-1) (n-5) h_1^3+n(n-3) h_1h_2+\ldots}
  \end{array}
\end{equation}
Observe that the coefficients for the various $h_\mu$ in $\DQcoinv[n](\mq)$ agree with those in the right-hand side of \pref{formule_hilb}, whenever $n\geq 5$.  This phenomenon seems to hold for $n$ larger then twice the order of terms calculated.     
  
Let us write $k(\mu)$ for the number of parts of a partition $\mu$, and denote by
 \begin{displaymath}
    d_\mu:=d_1! d_2! \cdots d_n!
      \end{displaymath}
  the product of the factorials of multiplicities of parts in $\mu$. Here $d_i$ is the multiplicity of the part $i$.
We then easily calculate that the coefficient of $h_\mu(\mq)$, in the right hand side of \pref{formule_hilb}, can be written in the form
    \begin{equation}
         \bleu{\frac{(n)_{k(\mu)}}{d_\mu}-\sum_{\nu} \frac{(n)_{k(\nu)}}{d_\nu}},
  \end{equation}
where the summation is over the set of partitions that can be obtained by removing one part of $\mu$. As usual, we denote by $(n)_k$ the product
  \begin{displaymath}
     \bleu{(n)_k:=n(n-1)\cdots (n-k+1)}.
  \end{displaymath}
  Another way of looking at all this is to say that the coefficient of a given $h_\mu$ stabilizes to a positive valued polynomial in $n$, as $n$ grows to be large enough.

\subsection*{Explicit values}
Explicit calculations give the following $h$-positive expressions,   in the case of the symmetric group.
\begin{equation}\label{valeurs_pour_Sn}
\begin{array}{rcl}
\DQcoinv[1]&=&1\\
\DQcoinv[2]&=&1+h_{{1}},\\
\DQcoinv[3]&=&1+2\,h_{{1}}+2\,h_{{2}},\\
\DQcoinv[4]&=&1+3\,h_{{1}}+3\,h_{{2}}+2\,{h_{{1}}}^{2}+5\,
h_{{3}},\\
\DQcoinv[5]&=&1+4\,h_{{1}}+4\,h_{{2}}+5\,{h_{{1}}}^{2}+4\,
h_{{3}}+10\,h_{{1}}h_{{2}}+14\,h_{{4}},\\
\DQcoinv[6]&=&1+5\,h_{{1}}+5\,h_{{2}}+9\,{h_{{1}}}^{2}+5\,
h_{{3}}+18\,h_{{1}}h_{{2}}+5\,{h_{{1}}}^{3}\\ &&\qquad\qquad+28\,h_{{3}}h_{{1}}+14\,{h_
{{2}}}^{2}+42\,h_{{5}}.
\end{array}
\end{equation}
From this it would be tempting (as we did in a first draft of this paper) to conjecture that $\DQcoinv[n]$ is always $h$-positive, but (as shown by recent calculations of H.~Blandin and F.~Saliola) this fails at $n=7$, indeed we get
\begin{eqnarray}
   \DQcoinv[7]&=&1+6\,h_{{1}}
                              +14\,{h_{{1}}}^{2}+6\,h_{{2}}
                               +14\,{h_{{1}}}^{3}+28\,h_{{1}}h_{{2}}+6\,h_{{3}}\\
                               &&\qquad
                               +42\,h_{{2}}{h_{{1}}}^{2}+14\,{h_{{2}}}^{2}+28\,h_{{1}}h_{{3}}+6\,h_{{4}}\\
                               &&\qquad
                            +84\,h_{{3}}h_{{2}}+84\,h_{{4}}h_{{1}}-36\,h_{{5}}+132\,h_{{6}}.
\end{eqnarray}
Observe the coefficient of $h_5$.
For the groups $W=G(2,n)$, which is the hyperoctahedral group $B_n$, we do have the $h$-positive expressions
\begin{eqnarray*}
\DQcoinv[G(2,2)]&=&1+2\,h_{{1}}+  h_{{2}}+{h_{{1}}}^{2} +2\,h_{{3}}+h_{{4}},\\
\DQcoinv[G(2,3)]&=& 1+3\,h_{{1}}+   2\,h_{{2}}+3\,{h_{{1}}}^{2} +
 3\,h_{{3}}+{h_{{1}}}^{3}+3\,h_{{1}}h_{{2}}  \\
 &&\quad +
 6\,h_{{3}}h_{{1}}+2\,h_{{4}} + 3\,h_{{1}
}h_{{4}}+5\,h_{{5}}  +6\,h_{{6}} +2\,h_{{7}}.
\end{eqnarray*}
Explaining when we do have $h$-positivity is still somewhat mysterious.

\section{Colored quasi-symmetric polynomials}\label{sec_colored}
In light of the results and conjecture considered above, we think it worthwhile to reformulate results obtained in \cite{aval_color} from this new perspective. Indeed, the relevant formulas take a new and much nicer format which gives indirect support to our conjecture,
since the Hilbert series considered happen to be provably $h$-positive. This was not noticed at the time of the writing of  \cite{aval_color}.

Let us consider the subspace of \monem{colored\,}\footnote{These where coined to be the $G(\ell,n)$-quasi-symmetric polynomials (or even $B$-quasi-symmetric when $\ell=2$) in~\cite{aval_color,BH}, but this terminology leads to confusion in the present context.} quasi-symmetric polynomials  of the space of diagonal $\S_n$-quasi-invariants (in the context of $X$ being $\ell\times n$ matrix of variables). 
Contrary to our previous presentation, colored quasi-symmetric polynomials are not defined as invariants. They are rather described in terms of a basis, indexed by ``colored composition''. Recall that 
{\em colored compositions} of length $p$ are $\ell\times p$-matrices 
 $$C= \begin{pmatrix}  
         c_{11} & c_{12} & \cdots & c_{1p}\\
         c_{21} & c_{22} & \cdots & c_{2p}\\
         \vdots &\vdots &\ddots &\vdots\\
         c_{\ell1} & c_{\ell2} & \cdots & c_{\ell p}\\
             \end{pmatrix}
 $$
 with non negative entries, and
such that the associated \monem{entries reading word}   (obtained by reading column by column from left to right, and each column from top to bottom) avoids the pattern of $\ell$ consecutive zeros.

To each colored composition $C$, we associate a \monem{monomial colored quasi-symmetric functions} by setting:
\begin{equation}\label{col_mon}
   \overline{M}_C:=\sum\prod_{1\le i\le \ell}\ \prod_{1\le k\le p} x_{i,a_{ik}}^{c_{ik}}
\end{equation}
where the sum is over all choices of $a_{ik}$ such that 
\begin{eqnarray*}
   a_{ik}&\le& a_{i+1,k},\qquad  {\rm when}\ 1\le i<\ell, \qquad  {\rm and}\\
    a_{\ell k}&<&a_{1,k+1},\qquad  {\rm for}\ 1\le k<p.
 \end{eqnarray*}
For example, we have
 $$  \overline{M}\Big[ \begin{smallmatrix}  
       1&3\\
       0&1\\
       2&0  
             \end{smallmatrix}\Big]=\sum_{a\le b<c \le d} x_{a}\,z_{b}^2\,x_{c}^3\,y_{d}.$$              
An example helps us point out the difference between diagonal $\S_n$-quasi-invariants and colored quasi-symmetric polynomials.
For $n=3$ and $\ell=2$, we have the three independent diagonal quasi-symmetric polynomials of degree $2$: 
$$\begin{array}{l}
   x_1y_1+x_2y_2+x_3y_3, \\
   x_1y_2+x_1y_3+x_2y_3,\qquad  {\rm and}\\
    x_2y_1+x_3y_1+x_3y_2,
 \end{array}$$
whereas we have only two
$$\begin{array}{l}
   x_1y_1+x_2y_2+x_3y_3+x_1y_2+x_1y_3+x_2y_3,\qquad  {\rm and}\\ 
    x_2y_1+x_3y_1+x_3y_2,
 \end{array}$$
colored quasi-symmetric polynomials in same degree.

We denote by $K$ the ideal generated by constant-term-free colored quasi-symmetric polynomials assuming that the parameters 
$n$ and $\ell$ are unambiguous from the context, and introduce the quotient of the ring of polynomials by the ideal $K$:
\begin{equation}\label{def_dqcoinv_col}
    \colored_n^{(\ell)}:=\poly/K.
  \end{equation}
The main result of \cite{aval_color} may be elegantly re-coined in terms of an $h$-positive expansion for the Hilbert series of  $\colored_n=\colored_n^{(\ell)}$. 
This new formulation has the extra advantage that it has similarities with the formula that we would expect to find under the hypothesis that conjecture~\ref{conj} holds.
In order to state this reformulation, we need to recall some notions concerning \monem{dyck paths}.
Recall that such a path is a sequence 
   $$\beta=p_0,p_1,\ldots ,p_{2\,n}$$
 of points $p_i=(x_i,y_i)$ in $\NN\times \NN$, with $x_i\leq y_i$,
$p_0=(0,0)$, $p_{2\,n}=(n,n)$, and such that either
     $$p_{i+1} = \begin{cases}
      p_i+(1,0)& \text{or}, \\[4pt]
      p_i+(0,1),\end{cases}$$
for all $i$.
We say that $n$ is the \monem{height} of $\beta$, and 
that we have a \monem{horizontal step} 
     $$s_i:=(p_{i-1},p_i),$$
  at \monem{level} $k\geq 1$, if
  $y_i=k=y_{i+1}$.
The set of Dyck paths is denoted by $\D_n$.
To any given Dyck path $\beta$, we associate the composition $\nu(\beta)$ obtained by counting  the number of level $k$ horizontal steps (ignoring the situation when this number is zero), for $k<n$.
An example is given in Figure~\ref{fig_chemin_dyck}, for the Dyck path
\begin{eqnarray*}
   \beta&=&\bleu{(0,0),(0,1),(0,2),\rouge{(1,2)}, (1,3),(1,4),\rouge{(2,4),(3,4),(4,4)},}\\
      &&\qquad\qquad  \bleu{(4,5),(4,6), \rouge{(5,6),(6,6)},(6,7),(6,8),(7,8),(8,8)}
   \end{eqnarray*}
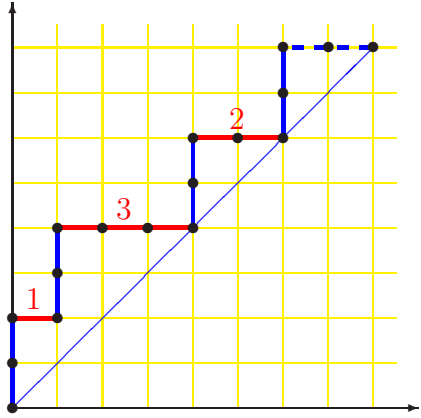
\begin{figure}[ht]\begin{center}\setlength{\unitlength}{6mm}
\begin{picture}(9,9)(0,0)
 \multiput(0,1)(0,1){8}{\jaune{\line(1,0){8.5}}}
  \multiput(1,0)(1,0){8}{\jaune{\line(0,1){8.5}}}
  \put(0,0){\vector(1,0){9}}
 \put(0,0){\vector(0,1){9}}
 \put(0,0){\bleu{\line(1,1){8}}}
  \linethickness{.5mm}
 \put(0,0){\bleu{\line(0,1){2}}}
  \put(1,2){\bleu{\line(0,1){2}}}
 \put(4,4){\bleu{\line(0,1){2}}}
 \put(6,6){\bleu{\line(0,1){2}}}
\put(0,2){\rouge{\line(1,0){1}}} \put(0.3,2.2){$\rouge{1}$}
\put(1,4){\rouge{\line(1,0){3}}} \put(2.3,4.2){$\rouge{3}$}
\put(4,6){\rouge{\line(1,0){2}}}  \put(4.8,6.2){$\rouge{2}$}
\multiput(6.2,8)(0.5,0){4}{\bleu{\line(1,0){.25}}}
 \put(0,0){\circle*{.25}} \put(0,1){\circle*{.25}} \put(0,2){\circle*{.25}} \put(1,2){\circle*{.25}} \put(1,3){\circle*{.25}}\put(1,4){\circle*{.25}}
 \put(2,4){\circle*{.25}} \put(3,4){\circle*{.25}} \put(4,4){\circle*{.25}} \put(4,5){\circle*{.25}} \put(4,6){\circle*{.25}} \put(5,6){\circle*{.25}}
 \put(6,6){\circle*{.25}} \put(6,7){\circle*{.25}} \put(6,8){\circle*{.25}} \put(7,8){\circle*{.25}} \put(8,8){\circle*{.25}}
 \end{picture}\end{center}
\caption{A Dyck path $\beta$ with $\nu(\beta)=\rouge{132}$.}
\label{fig_chemin_dyck}
\end{figure}
With these notions at hand, we may now give our new formula.

\begin{proposition}\label{theo_col}\label{prop_col}
The Hilbert series of the quotient $\colored_n$ is given by the formula:
\begin{equation}\label{eq_hilb_col}
   \bleu{\colored_n(\mq)=\sum_{\beta\in \D_n} h_{\nu(\beta)}(\mq)},
\end{equation}
whose dependence on $\ell$ is entirely encapsulated in the number of variables  in $\mq$.
\end{proposition}
Observe that if we set all the $q_i$ equal to $1$, we get a combinatorial expression  which is interesting on its own:
\begin{equation}\label{eq_comb_col}
\sum_{\beta\in\D_n}\,\prod_{k\in \nu(\beta)}\binom{k+\ell-1}{ k} = {1\over \ell\, n+1}\binom{(\ell+1)\,n}{ n},
\end{equation}   
in which one may consider $\ell$ as a variable, hence we actually get a polynomial identity\footnote{It is a well-known fact that the right-hand-side is actually a polynomial in $\ell$.}.  For integral values of $\ell$, equation \pref{eq_comb_col} may be proven via a simple bijection on paths.
Observe also that both spaces $ \colored_n$ and $\DQcoinv[n]$ coincide when $\ell=1$.

\section{Proofs}

\subsection{Theorem \ref{mainthm}} Recall that we are asserting here that there exists a universal expression for the Hilbert series of $\DQcoinv$ of the form
    \begin{equation}\label{schur_generic_bis}
              \bleu{\DQcoinv(\mq) = \sum_{\mu} c_\mu\, s_\mu(\mq),\qquad c_\mu\in\NN,}
       \end{equation}
with the sum running over partitions $\mu$ of integers $d\leq 2\,r\,n-r-n$, each such partitions having at most $n$ parts. 
This restriction on the number of parts follows from the fact that this holds for the whole space $\mathcal{R}_n$, of which $\DQcoinv$ (or rather $\DQW$) can be considered as a subspace.

In order to prove inequality (\ref{interval_entiers}), we need to introduce some notations. The sum of all the entries a $r$-matrix $A$, divided by $r$, is an integer that we denote by $w_r(A)$. This is said to be the \monem{$r$-size} of $A$. To avoid ambiguity, we avoid $r$-matrices having their last column vanishing. The number of column of such a $r$-matrix is its \monem{length}.

Given  a $r$-vector $V$ (a single-column $r$-matrix), there is a lexicographically largest $r$-matrix  $\bleu{A(V)}$ 
such that
\begin{itemize}\itemsep=4pt
 \item  all of the columns of $A(V)$ are of $r$-size $1$,
 \item the sum of the columns of $A(V)$ is $V$,
\item the columns of $A(V)$ occur in decreasing  lexicographic\footnote{Considering entries from top to bottom.} order from left to right.
\end{itemize}
We denote $\bleu{\theta(V)}$ the \monem{first column} of $A(V)$, and set $\bleu{\Delta(V)}:=V-\theta(V)$.
For $V^\tr=(2,3,4)$, we get 
    $$A(V)=\begin{pmatrix} 2 & 0 & 0\\
                                          1 & 2 & 0\\
                                          0 & 1 & 3
                  \end{pmatrix},$$
hence $\theta(V)^\tr=(2,10)$ and $\Delta(V)^\tr=(0,2,4)$.
We now associate to any $r$-vector $V$ the smallest set, denoted by $S(V)$,  of $r$-matrices that contains $A(V)$ and that is closed under the operation that consists in taking sum of consecutive columns. For example, still with $V^\tr=(2,3,4)$, we have
   $$S(V):=\left\{\begin{pmatrix} 2 & 0 & 0\\
                                          1 & 2 & 0\\
                                          0 & 1 & 3
                  \end{pmatrix}, \begin{pmatrix}
                                          2 & 0\\
                                          3 & 0\\
                                          1 & 3
                  \end{pmatrix}, \begin{pmatrix} 2 & 0\\
                                          1 & 2\\
                                          0 & 4
                  \end{pmatrix} , \begin{pmatrix} 2\\
                                          3\\
                                         4
                  \end{pmatrix}  \right\}$$
For two $\ell$-row matrices, $A=A_1\cdots A_k$ and $B=B_1\cdots B_j$, the \monem{concatenation} $\bleu{AB}$ is the matrix having columns
   $$\bleu{AB:=A_1\cdots A_kB_1\cdots B_j}.$$
For a general $r$-composition matrix  $A=A_1\,A_2\dots A_k$, we define $S(A)$ to be the set obtained by all possible concatenation of matrices successively picked from each of the sets $S(A_i)$. 

We now come to the definition of the polynomials $G[A]:=G[A](X)$ that are used to prove (\ref{interval_entiers}).
These are indexed by \monem{trans $r$-matrices}, which is to say $r$-matrices $A=A_1\,A_2\dots A_k$ for which there exists $1\le j\le k$ such that $w_r(A_1\dots A_j)\ge j$.
It is clear that any $r$-composition matrix is trans. Let $W=G(r,n)$.
\begin{definition}\label{def:G}
   To a trans $r$-matrix $A$, we associate the $W$-quasi-invariant polynomial, $\bleu{G[A]=G_A(X)}$ recursively defined as follows.
\begin{itemize}
\item If $A$ is a $r$-composition, we set 
     $$\bleu{G_A:=\sum_{B\in S(A)} M_B}.$$
\item If not, there is a unique column decomposition of $A$ as a concatenation 
     $$A=B\,\bz\, VC,$$
where $B$ a $r$-matrix (say of length $j$), $V$ is a non-zero $r$-vector, and $C$ is a $r$-composition.
We then set 
    $$\bleu{G[A]:= G[{B\,V C}] - X_{j+1}^{\theta(V)} \,G[B\Delta(V)\,C]}.$$
\end{itemize}
It should be clear that when $A=B\,\bz\, VC$ is trans, then so are both $B\,V\,C$ and $B\Delta(V)C$. Thus, the family $G[A]$ is well-defined by induction on the length of $A$.
\end{definition}
It is helpful to consider an explicit an example. With $r=2$ and $n=4$, we compute that
\begin{eqnarray*}
         \Gpol{ 0&2&0&1\\   0&2&0&3}
 &=&  \Gpol{0&2&1\\  0&2&3}
             -x_{3}\,y_{3}\,   \Gpol{0&2&0\\   0&2&0}\\
&=&  \Gpol{2&1\\   2&3}
             -x_{1}^2  \Gpol{0&1\\ 2&3}
             -x_{3}\,y_{3} \left(\Gpol{2&0\\ 2&2}
             -x_{1}^2\, \Gpol{0&0\\ 2&2}\right).
\end{eqnarray*}
Since
\begin{eqnarray*}
 \Gpol{2&1\\   2&3}&=&\Mpol{2&1\\ 2&3}+\Mpol{2&0&1\\ 2&0&3}+\Mpol{2&1&0\\ 2&1&2}+\Mpol{2&0&1&0\\ 0&2&1&2},\\ 
\Gpol{0&1\\ 2&3}&=&\Mpol{0&1\\ 2&3}+\Mpol{0&1&0\\ 2&1&2},\\ 
\Gpol{2&0\\ 2&2}&=&\Mpol{2&0\\ 2&2}+\Mpol{2&0&0\\ 0&2&2},\\ 
\Gpol{0&0\\ 2&2}&=&\Mpol{0&0\\ 2&2},
\end{eqnarray*}
we finally get
$$\Gpol{ 0&2&0&1\\   0&2&0&3}=
x_{2}^2\, y_{2}^2\,x_{4}\,y_{4}^3+x_{2}^2\,x_{3}\,y_{3}^3\,y_{4}^2+x_{2}^2\,y_{3}^2\,x_{ 4}\,y_{4}^3+x_{3}^2\,y_{3}^2\,x_{4}\,y_{4}^3-x_{3}^3\,y_{3}^3\,y_{4}^2.$$
Observe that the lexicographic order leading monomial of $\Gpol{0&2&0&1\\ 0&2&0&3}$  is precisely 
   $$X^{\big(\begin{smallmatrix} 0201\\ 0203\end{smallmatrix}\big)}=x_{2}^2\,y_{2}^2\,x_{4}\,y_{4}^3.$$
This is shown to hold in full generality in the following proposition.
\begin{proposition}\label{prop:Glex}
For any trans $r$-matrix $A$, the leading monomial of $G_A(X)$ is $X^A$.
\end{proposition}
\noindent
Proposition \ref{prop:Glex} is established through the next two lemmas.
\begin{lemma}\label{lem:recG}
          For any $r$-composition matrix $A$, we have  
   \begin{equation}\label{eq:lemrecG}
         G_{\bz A}=G_A(X_{-1}),
     \end{equation}
writing $X_{-1}$ for the alphabet obtained from $X$ by removing its first column of variables
\end{lemma}
\begin{proof}[\bf Proof.]
We write $A=VC$, with $V$ the $r$-vector corresponding to the first column of $A$. In view of Definition \ref{def:G}, we have 
\begin{equation}
G_{VC} = X^{\theta(V)} G_{\Delta(V) C} + G_{VC}(X_{-1})
\end{equation}
which implies \pref{eq:lemrecG}.
\end{proof}
\begin{lemma} 
   Let $A$ be a $r$-matrix of length $j$, and $D$ a $r$-composition matrix, then we have
       \begin{equation}\label{eq:Glex}
                G_{AD}(X)=X^{A}G_{\bz^j D}(X)+({\rm terms}<_{\rm lex} X^A).
      \end{equation}
\end{lemma}
\begin{proof}[\bf Proof.]
If $A$ is a $r$-composition, \pref{eq:Glex} is a direct consequence of Definition \ref{def:G}. 
If not, equation \pref{eq:Glex} is shown to hold by induction on the length of $A$. 
Consider the unique factorization
            $$A=B\bz V C$$ 
 with $B$ is a $r$-matrix of length $j$, $V$ a non-vanishing $r$-vector, and $C$ is a $r$-composition.
We use Definition \ref{def:G} and Lemma \ref{lem:recG} to calculate that
\begin{eqnarray*}
   G_{B\bz V CD}&=&G_{BV CD} - X^{\bz^{j}\theta(V)\bz^{n-j-1}} G_{B\Delta(V)CD}\\ 
         &=&X^B\,G_{\bz^j V CD}-X^{\bz^{j-1}\theta(V)\bz^{n-j}}  X^B\,G_{\bz^j \Delta(V)CD}+({\rm terms}<_{\rm lex}X^{B})\\ 
        &=&X^{B}\,G_{\bz^{j+1} V CD}+({\rm terms}<_{\rm lex}X^{B})\\ 
       &=&X^{B}\,G_{\bz^{j} V CD}(X_{-1})+({\rm terms}<_{\rm lex}X^{B})\\ 
       &=&X^{B\bz V C}\,G_{\bz^{p-1} D}(X_{-1})+({\rm terms}<_{\rm lex}X^{B\bz V C})\\ 
       &=&X^{B\bz V C}\,G_{\bz^p D}+({\rm terms}<_{\rm lex}X^{B\bz V C}).
\end{eqnarray*}
\end{proof}
\begin{proof}[\bf Proof of Condition \pref{interval_entiers}.]  
The proof is now an easy consequence of the following two observations.
\begin{itemize}
\item Any $r$-matrix $A$ with $w_r(A) = n$ is trans. Thus any monomial $X^A$ with $A$ a $r$-matrix and $w_r(A)= n$ is the leading monomial of in ideal generated  by quasi-invariant polynomials for the group $G(n,r)$.
\item Any monomial of total degree strictly greater than $2rn-r-n$ is the multiple of a monomial $X^A$, with $A$ a $r$-matrix and for which $w_r(A)= n$. Since it is true for any monomial
of such degree, any monomial of degree strictly greater than $2rn-r-n$ lies in the ideal, whence \pref{interval_entiers}.
\end{itemize}
\end{proof}

\subsection{Colored polynomials}
The main aim of this subsection is to prove Proposition~\ref{prop_col}. 
Let us first show that  formula~\pref{eq_comb_col} holds for all
positive integer values of $\ell$, hence it follows that we have a polynomial identity.
Recall that an \monem{$\ell$-path} is a finite sequence of points $p_i=(x_i,y_i)$
in the plane such that $p_0=(0,0)$ and
     $$p_{i+1} = \begin{cases}
     p_i+(\ell,0)& \text{or}, \\[4pt]
     p_i+(0,1).\end{cases}$$
We say that we have an \monem{$\ell$-Dyck path} if the path satisfies the further condition that $x_i\le y_i$ for all $i$.
Let us denote by $\D_n^{(\ell)}$ the set of $\ell$-Dyck paths of height $n\,\ell$.
It is well-known  that $\ell$-Dyck paths are enumerated by the \monem{Fuss-Catalan numbers}:
\begin{equation}
   \bleu{\# \D_n^{(\ell)} = {1\over\ell\, n+1}\binom{(\ell+1)n}{n}},
\end{equation}
appearing as the right-hand side of \pref{eq_comb_col}.
To relate this to the left-hand-side of \pref{eq_comb_col}, we consider
 \monem{$\ell$-colored Dyck paths} of height $n$, which are simply height $n$ Dyck path $\beta$ whose horizontal 
steps, at levels $k<n$, have been \monem{colored} by elements of the set $\{1,2,\dots,\ell\}$. Here, we assume that the colors of steps on a same level
are weakly increasing from left to right, according to the color order $1<2<\ldots<\ell$. In other words, the coloring is a function $\gamma $, associating to each horizontal step $s_i$ of $\beta$, a color $\gamma(s_i)$, in such a manner that
    $$\gamma(s_i)\leq \gamma(s_{i+1}),$$
 if both $s_i$ and $s_{i+1}$ are horizontal steps, hence inevitably at the same level.
Thus $\ell$-colored Dyck paths are pairs $(\beta,\gamma)$, consisting of a path  with its coloring.

We establish formula~\pref{eq_comb_col}  by building a bijection between
$\ell$-colored Dyck paths of height $n$,
and $\ell$-Dyck paths of height $n\,\ell$.
Given an $\ell$-colored Dyck path $(\beta,\gamma)$, 
we denote by $a_{kj}$  the number of level $k$ horizontal steps of color $j$,
and we iteratively construct a path 
\begin{equation} \label{defn_Phi}
    \Phi(\beta,\gamma)=q_0,q_1\ldots,
\end{equation}
 using this data. Starting with $q_0=(0,0)$,  and running through the $a_{kj}$'s as $k$ goes from $1$ to $n-1$ and $j$ goes from $1$ to $\ell$, we successively add to $\pi$
\begin{itemize}
    \item $a_{kj}$ horizontal steps  of length $\ell$, followed by  
    \item one vertical step $(0,1)$.
\end{itemize}
One easily check that this indeed results in an $\ell$-Dyck path of height $n\,\ell$.

This transformation is illustrated in Figure~\ref{fig_Phi} with $\{\rouge{\rm red},\vert{\rm green}\}$ as color set  (i.e.: $\ell=2$).
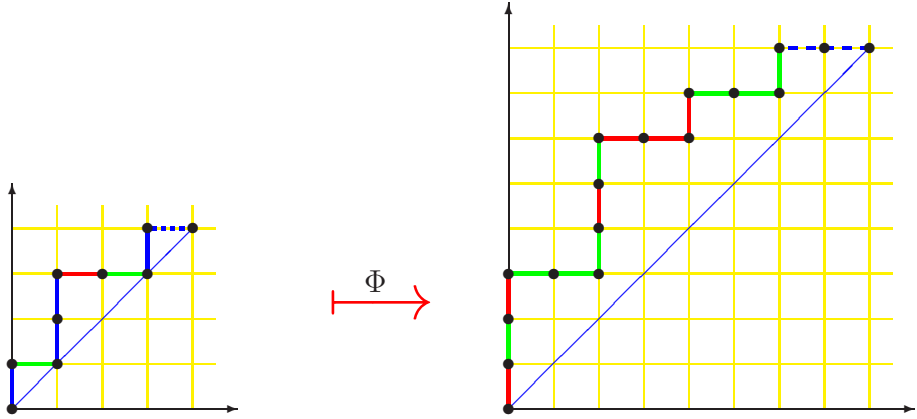
\begin{figure}[ht]\begin{center}\setlength{\unitlength}{6mm}
\begin{picture}(18,10)(1,0)
 \multiput(0,1)(0,1){4}{\jaune{\line(1,0){4.5}}}
  \multiput(1,0)(1,0){4}{\jaune{\line(0,1){4.5}}}
  \put(0,0){\vector(1,0){5}}
 \put(0,0){\vector(0,1){5}}
 \put(0,0){\bleu{\line(1,1){4}}}
{  \linethickness{.4mm}
 \put(0,0){\bleu{\line(0,1){1}}}
 \put(0,1){\vert{\line(1,0){1}}}
 \put(1,1){\bleu{\line(0,1){2}}}
 \put(1,3){\rouge{\line(1,0){1}}}
 \put(2,3){\vert{\line(1,0){1}}}
 \put(3,3){\bleu{\line(0,1){1}}}
\multiput(3.1,4)(0.2,0){5}{\bleu{\line(1,0){.1}}}
 \put(0,0){\circle*{.25}} \put(0,1){\circle*{.25}} \put(1,1){\circle*{.25}} \put(1,2){\circle*{.25}} \put(1,3){\circle*{.25}}\put(2,3){\circle*{.25}}
 \put(3,3){\circle*{.25}} \put(3,4){\circle*{.25}} \put(4,4){\circle*{.25}} 
\put(7.8,2.6){$\Phi$}
\put(7,2){\rouge{\Huge$\longmapsto$}}
}
 \multiput(11,1)(0,1){8}{\jaune{\line(1,0){8.5}}}
  \multiput(12,0)(1,0){8}{\jaune{\line(0,1){8.5}}}
 \put(11,0){\vector(1,0){9}}
 \put(11,0){\vector(0,1){9}}
 \put(11,0){\bleu{\line(1,1){8}}}
{  \linethickness{.5mm}
 \put(11,0){\rouge{\line(0,1){1}}}
 \put(11,1){\vert{\line(0,1){1}}}
 \put(11,2){\rouge{\line(0,1){1}}}
 \put(11,3){\vert{\line(1,0){2}}}
 \put(13,3){\vert{\line(0,1){1}}}
  \put(13,4){\rouge{\line(0,1){1}}}
  \put(13,5){\vert{\line(0,1){1}}}
 \put(13,6){\rouge{\line(1,0){2}}}
 \put(15,6){\rouge{\line(0,1){1}}}
 \put(15,7){\vert{\line(1,0){2}}}
 \put(17,7){\vert{\line(0,1){1}}}
  \multiput(17.2,8)(0.4,0){5}{\bleu{\line(1,0){.2}}}
  \put(11,0){ \put(0,0){\circle*{.25}} \put(0,1){\circle*{.25}} \put(0,2){\circle*{.25}} \put(0,3){\circle*{.25}} \put(1,3){\circle*{.25}}\put(2,3){\circle*{.25}}
 \put(2,4){\circle*{.25}} \put(2,5){\circle*{.25}} \put(2,6){\circle*{.25}} \put(3,6){\circle*{.25}} \put(4,6){\circle*{.25}} \put(4,7){\circle*{.25}}
 \put(5,7){\circle*{.25}} \put(6,7){\circle*{.25}} \put(6,8){\circle*{.25}} \put(7,8){\circle*{.25}}\put(8,8){\circle*{.25}}
 }
}
\end{picture}\end{center}
\caption{The transformation $\Phi$.}
\label{fig_Phi}
\end{figure}
It is easy to check that $\Phi$ is indeed a bijection. \qed

With the intention of giving a proof of Proposition~\ref{prop_col}, let us recall the following result of~\cite{aval_color} which generalizes the main result of \cite{ABB}.
\begin{lemma}\label{prop_aval_color}
A monomial basis of the quotient $\colored_n^{(\ell)}$ is given by the monomials $X^A$ such that $\pi(A)$ is an $\ell$-Dyck path.
\end{lemma}

This last statement uses the following ``encoding'' of monomials $X^A$ in terms of lattice paths. The lattice path  $\pi(A)$ is obtained by applying the following construction to  entries-reading-word $w(A)$ of the $\ell\times n$ exponent matrix $A$.
Starting with the point $(0,0)$, for each entry $a$ of $w(A)$ we add $a$ horizontal steps $(\ell,0)$, followed by one vertical step $(0,1)$.
For instance,   Figure~\ref{fig_piA} represents the path $\pi(A)$ associated to the monomial
   $$X^{\big(\begin{smallmatrix}        1&1&0\\      0&2&1    \end{smallmatrix}\big)}=x_1\,x_2\,y_2^2\,y_3,$$
whose exponent matrix has entries-reading-word $101201$,  
 Observe that each horizontal step is of length two.
One associates variables to each level, namely $x_i$ at even level $2\,(i-1)$, and $y_i$ at odd level $2\,i-1$.
We then  readout the monomial from the path by the simple device
     of associating as exponent of the variable of a level, the number of horizontal steps on that level.
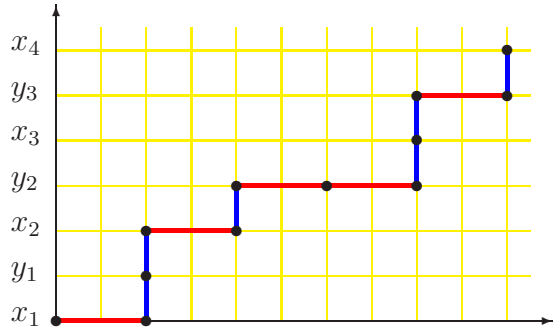
\begin{figure}[ht]\begin{center}\setlength{\unitlength}{6mm}

\begin{picture}(8,8)(0,0)
 \multiput(0,1)(0,1){6}{\jaune{\line(1,0){10.5}}}
  \multiput(1,0)(1,0){10}{\jaune{\line(0,1){6.5}}}
  \put(0,0){\vector(1,0){11}}
 \put(0,0){\vector(0,1){7}}
  \linethickness{.5mm}
 \put(0,0){\rouge{\line(1,0){2}}}
 \put(2,0){\bleu{\line(0,1){1}}}
 \put(2,1){\bleu{\line(1,0){0}}}
 \put(2,1){\bleu{\line(0,1){1}}}
 \put(2,2){\rouge{\line(1,0){2}}}
 \put(4,2){\bleu{\line(0,1){1}}}
 \put(4,3){\rouge{\line(1,0){4}}}
 \put(8,3){\bleu{\line(0,1){1}}}
 \put(8,4){\bleu{\line(1,0){0}}}
 \put(8,4){\bleu{\line(0,1){1}}}
 \put(8,5){\rouge{\line(1,0){2}}}
 \put(10,5){\bleu{\line(0,1){1}}}
 \put(0,0){\circle*{.25}} \put(2,0){\circle*{.25}} \put(2,1){\circle*{.25}}\put(2,2){\circle*{.25}}
 \put(4,2){\circle*{.25}} \put(4,3){\circle*{.25}}\put(6,3){\circle*{.25}}\put(8,3){\circle*{.25}}
 \put(8,4){\circle*{.25}} \put(8,5){\circle*{.25}} \put(10,5){\circle*{.25}}\put(10,6){\circle*{.25}}
 \put(-1,0){$x_1$} \put(-1,1){$y_1$}
\put(-1,2){$x_2$} \put(-1,3){$y_2$}
\put(-1,4){$x_3$} \put(-1,5){$y_3$}
\put(-1,6){$x_4$}  
 \end{picture}\end{center}
\caption{Example of $\pi(A)$ for $\ell=2$.}
\label{fig_piA}
\end{figure}

\begin{proof}[\bf Proof of Proposition~\ref{prop_col}.]  
For a given  Dyck path $\beta$, let us consider the set
$\mathcal{C}(\beta)$  of $\ell$-Dyck path $\pi$ such that there exists a coloring $\gamma$ with 
$\Phi(\beta,\gamma)=\pi$. In formula,
    $$\mathcal{C}(\beta)=\{\ \pi\ |\ \exists \gamma\ {\rm such\ that}\ \Phi(\beta,\gamma)=\pi\},$$
  with $\Phi$ as in~\pref{defn_Phi}.
If $a_k$ is the number of horizontal steps at level $k$ in $\beta$, the choice of $\ell$-coloring is equivalent to the choice of a monomial 
           $$x_{k1}^{a_{k1} }x_{k2}^{a_{k2}} \,\cdots\, x_{k\ell}^{a_{k\ell}} ,$$
 with $a_{kj}$ giving the number of steps getting to be colored $j$, hence $a_k=a_{k1}+\ldots+a_{k\ell}$.
The Hilbert series of the resulting set of monomial is $h_{a_k}(\mq)$.
Since there is independence in the choice of colorings at different levels,  the Hilbert series of the monomials associated to $\ell$-Dyck paths in $\mathcal{C}(\beta)$
is $h_{\nu(\beta)}(\mq)$.
The fact that $\Phi$ is a bijection gives the proof of Proposition~\ref{prop_col}, in view of Lemma~\ref{prop_aval_color}.
\end{proof}
\section{Open problems}
The main  remaining open question in all the above considerations is to find a explicit (even conjectural) candidate for partitions $\mu(\beta)$, one 
for each Dyck path $\beta$, which would explain the $h$-positive expansion in Proposition~\ref{mainprop}. Naturally, similar questions may be stated whenever the universal Hilbert series $\DQcoinv(\mq)$, for a group $W$, happens to be $h$-positive. As discussed in the paper, the resulting entirely combinatorial description of  the universal $\DQcoinv(\mq)$ would give, in one compact formula, the $GL_\ell$-action characters for all the spaces  $\DQcoinv^{\ell}$.

\end{document}